\documentclass[12pt,a4paper,reqno]{amsart}
\usepackage{amsmath}
\date{}
 \usepackage[latin1]{inputenc}
\usepackage{amssymb , amsmath, amsthm}
 \usepackage{latexsym, esint}
 \usepackage{ulem}

\usepackage{latexsym}

\usepackage{color}

\newtheorem{thm}{Theorem}[section]

\newtheorem{lemma}{Lemma}[section]

 \newcommand{\R}{\mathbb R}
\newcommand{\Sp}{\mathbb S}
\newcommand{\C}{\mathbb C}

\begin{document}

\title [ First eigenvalue of the Sub-Laplacian on  CR spheres]{  The first positive eigenvalue of the Sub-Laplacian on  CR spheres}
\author{ Amine Aribi }

\author{Ahmad El Soufi}
\address{Universit\'e de Tours, Laboratoire de Math\'ematiques
et Physique Th\'eorique, UMR-CNRS 7350, Parc de Grandmont, 37200
Tours, France.} \email{Amine.Aribi@lmpt.univ-tours.fr,   ahmad.elsoufi@lmpt.univ-tours.fr}

\begin{abstract}
We prove that the   first positive eigenvalue, normalized by the volume,  of the sub-Laplacian associated with a strictly pseudoconvex pseudo-Hermitian structure $\theta$ on the CR sphere  $\Sp^{2n+1}\subset \C^{n+1}$, achieves its maximum when $\theta$ is the standard contact form. 

\end{abstract}

\subjclass[2010]{32V20, 35H20,  58J50.}
\keywords{CR Sphere, Sub-Laplacian, eigenvalue}

\maketitle
\section{Introduction and statement of the main result}
According to a classical result of Hersch \cite{Hersch},  given any Riemannian metric $g$ on the 2-dimensional sphere $\Sp^2$, the first positive eigenvalue $\lambda_1 (g)$ of the Laplace-Beltrami operator  $\Delta_g$ satisfies the estimate 
\begin{equation}\label{hersch}
\lambda_1 (g) A(g)\le  \lambda_1 (g_0) A(g_0)
\end{equation}
where $g_0$ is the standard metric of $\Sp^2$ and $A(g)$ is the area of $\Sp^2$ with respect to $g$. Moreover, the equality holds in \eqref{hersch} if and only if $g$ is isometric to $g_0$. This result was extended to higher dimensional spheres by Ilias and the second author as follows  (see \cite[proposition 3.1]{SoIl}) : If  a Riemannian metric $g$ on the $n$-dimensional sphere $\Sp^n$ is conformal to the standard metric $g_0$, then
\begin{equation}\label{elsoufi-ilias}
\lambda_1 (g) V(g)^{\frac2n}\le  \lambda_1 (g_0) V(g_0)^{\frac2n}
\end{equation}
where $V(g)$ denotes the Riemannian volume of the sphere with respect to $g$. Again, the equality holds in \eqref{elsoufi-ilias} if and only if $g$ is isometric to  $g_0$. 

\smallskip
The aim of the present paper is to establish a version of the estimate \eqref{elsoufi-ilias} for the first positive eigenvalue of the sub-Laplacian on the   CR sphere $\Sp^{2n+1}\subset \C^{n+1}$. Indeed, let  
$$\theta_0=\frac{i}{2}\sum_{j=1}^{n+1}\left(\zeta_j\,d\bar{\zeta}_j- \bar{\zeta}_j\,d\zeta_j\right)$$
 be the standard contact form on $\Sp^{2n+1}$ whose kernel coincides with the Levi distribution $H(\Sp^{2n+1})=T\Sp^{2n+1}\cap JT\Sp^{2n+1}$, where $J$ is the complex structure of $\C^{n+1}$. The set ${\mathcal P}_+(\Sp^{2n+1})=\{f\theta_0\ ; \ f\in C^\infty(\Sp^{2n+1}) \mbox{ and } f>0\}$ contains all  pseudo-Hermitian structures on $\Sp^{2n+1}$ whose Levi form is positive definite.  Given a  pseudo-Hermitian structure $\theta\in {\mathcal P}_+(\Sp^{2n+1})$, we  denote by  $\lambda_1( {\theta})$ the first positive eigenvalue of the corresponding sub-Laplacian $\Delta_{\theta} $, and by $V({\theta})$ the volume of $\Sp^{2n+1}$ with respect to the volume form $\psi_{{\theta}}=\frac{1}{n!2^n } \, {\theta}\wedge (d{\theta})^n$  (see the next section for precise definitions). The main result of this paper is 

\begin{thm}\label{theo 41}
For every pseudo-Hermitian structure $\theta\in {\mathcal P}_+(\Sp^{2n+1})$ we have 
\begin{equation}\label{maineq}
\lambda_1({\theta})V( {\theta})^{\frac{1}{n+1}}\leq \lambda_1({\theta_0})V( {\theta_0})^{\frac{1}{n+1}} .
\end{equation}
The  equality holds  in \eqref{maineq}  if and only if there exists a CR-automorphism $\gamma$ of $ \mathbb{S}^{2n+1}$ such that ${\theta}=c \, \gamma^*\theta_0 $ for some constant $c>0$,
 or if and only if there exist $p\in \mathbb{S}^{2n+1}$ and  $t\ge 0$  such that 
$${\theta} =\frac{c}{\left|\cosh t +\sinh t \, (\zeta, p)\right|^2}\,\theta_{0},$$
where $(\ ,\ )$ denotes the standard Hermitian product of $\C^{n+1}$.
\end{thm}

This result can be seen as a contribution to the program aiming to recovering the main results of spectral geometry, established for the eigenvalues of the Laplace-Beltrami operator on a compact Riemannian manifold,  in the realm of CR and pseudo-Hermitian geometry. This program has motivated a lot of research in recent years and we can find significant contributions in   \cite{ArDrSou, ArDrSou2,  ArDrSou1, AriSou,  BarlettaDragomirSpectrum, Barletta, BarlettaDragomirSublaplacians, ChangLi,  Changchiu, Cowling, Greenleaf, hako,  IvVss2, IvVss1, Kok, LiLuk,  NiuZhang, ponge, NS}.


\section{Proof of  Theorem \ref{theo 41}}
Let $M$ be a connected differentiable manifold of  dimension $2n+1\ge 3$.
A CR structure on $M$ is a couple $(H(M),J)$ where 
$H(M)$ is a $2n$-dimensional subbundle of the tangent bundle $TM$, the so-called  Levi distribution, endowed with a pseudo-complex operator $J$ satisfying the following integrability condition : $\forall X,Y\in\Gamma(H(M))$,
$$[X,Y]-[J X,JY] \ \in \ \Gamma(H(M))$$
and
$$J\left([X,Y]-[J X,JY]\right)= [J X,Y]+[X,J Y] .$$
Real hypersurfaces of $\C^{n+1}$ are the most natural examples of  CR manifolds. Indeed, if $M\subset  \C^{n+1}$ is such a hypersurface, then $H(M):= T M \cap J(T M)$ endowed with the restriction of the standard complex structure $J$ of $\C^{n+1}$, is a CR structure on $M$. 

\smallskip
If  $(M, H(M), J)$ is an orientable CR manifold, then there exists a nontrivial 1-form $\theta\in \Gamma(T^* M)$ such that $Ker\theta = H(M)$. Such a 1-form, called  {pseudo-Hermitian} structure,  is  of course not unique. The set of pseudo-Hermitian structures  consists in all the forms $f\theta$, where $f$ is a smooth nowhere zero function on M.  
To each  pseudo-Hermitian structure  $\theta$ we associate its {Levi form} $L_{\theta}$ defined on $H(M)$  by
    \begin{equation*}
        L_{\theta}(X,Y)=-d\theta(JX,Y)=\theta([J X,Y]).
    \end{equation*} 
The integrability of $J$ implies that $L_{\theta}$ is symmetric and $J$-invariant. The $CR$ manifold $M$ is called  {strictly pseudoconvex} if $L_{\theta}$ is  definite.  Of course, this condition does not depend on the choice of $\theta$ (since $L_{f\theta}=fL_{\theta}$). In  the sequel, we  denote by ${\mathcal P}_+(M)$ the set of all pseudo-Hermitian structures with positive definite Levi form. Every $\theta\in {\mathcal P}_+(M)$ is in fact a contact form which induces on $M$ the following volume form 
$$\psi_{\theta}=\frac{1}{2^n\ n!} \ \theta\wedge (d\theta)^n .$$ 
The associated divergence $\text{div}_\theta$ is defined, for every  smooth vector field $Z$ on $M$, by
$$\mathcal{L}_Z \psi_{\theta}= \text{div}_\theta(Z)\,  \psi_{\theta}.$$
The sub-Laplacian $\Delta_\theta$ is then defined for all $u\in C^\infty(M)$, by
$$\Delta_\theta u=\text{div}_\theta(\nabla^H u)$$
where $\nabla^H u\in \Gamma(H(M))$ is the horizontal vectorfield such that, $\forall X\in H(M)$, $du(X)=L_\theta (\nabla^H u,X)$. The following integration by parts formula holds for any $u,v\in C_0^\infty(M)$:
$$\int_M (\Delta_\theta u)\, v \,\psi_\theta=-\int_M L_\theta(\nabla^H u,\nabla^H v)\, \psi_\theta.$$
Given $\theta\in {\mathcal P}_+(M)$, there is a unique vectorfield $\xi$, often called  Reeb vectorfield,   that satisfies  $\theta(\xi)=1$ and $\xi\rfloor d\theta=0$. The Levi form $L_\theta$ extends to a Riemannian metric on $M$  (the Webster metric)   given by
    $$g_\theta (X,Y)= L_\theta(X^H,Y^H) + \theta (X)\theta (Y)$$
    with $X^H=X-\theta(X)\xi$. The corresponding Laplace-Beltrami operator $\Delta_{g_\theta}$ is related to $\Delta_\theta$ by the following formula (see \cite{Greenleaf})
    \begin{equation}\label{green}
    \Delta_\theta =\Delta_{g_\theta} -\xi^2.
        \end{equation}
The sub-Laplacian $\Delta_\theta$ is a sub-elliptic operator of order $1/2$.  When $M$ is compact, it admits  a self-adjoint  extension  to an unbounded operator   of $L^2(M, \psi_{\theta})$ whose resolvent is compact (see for instance \cite[Lemma 2.2]{ArDrSou2}). Hence, the spectrum of $-\Delta_\theta$  is discrete and consists of a sequence of nonnegative eigenvalues of finite multiplicity $\{\lambda_k(\theta)\}_{k\ge 0}$ with $\lambda_0(\theta)=0$. The min-max variational principle  gives 
    \begin{equation}\label{minmax0}
    \lambda_1(\theta)=\inf_{\int_M u \, \psi_{\theta}=0} \frac{\int_M\vert \nabla^H u\vert_{\theta}^2\, \psi_{\theta}}{\int_M u^2\, \psi_{\theta}}
        \end{equation}
where $\vert \nabla^H u\vert_{\theta}^2=L_\theta(\nabla^H u,\nabla^H u)$.

\subsection{The CR Sphere}
Let $\mathbb{S}^{2n+1}$ be the unit Sphere in $\mathbb{C}^{n+1}$  
$$\mathbb{S}^{2n+1}=\left\{\zeta=(\zeta_1,...,\zeta_{n+1})\in \mathbb{C}^{n+1}\, ; \,\sum_{j\le n+1}\vert \zeta_j\vert^2 =1\right\}$$ 
endowed with its standard   CR-structure.  
The restriction to $\mathbb{S}^{2n+1}$ of the contact form 
$$\theta_0=-\frac{i}{2}\sum_{j=1}^{n+1}\left(\bar{\zeta}_j\,d\zeta_j-\zeta_j\,d\bar{\zeta}_j\right)$$
is a pseudohermitian structure whose Reeb field 
$\xi=i\sum_{j=1}^{n+1}\left(\zeta_j\frac{\partial}{\partial \zeta_j}-\bar{\zeta}_j\frac{\partial}{\partial \bar{\zeta}_j}\right)$
generates the natural action 
of $\, \mathbb{S}^1$ on $\mathbb{S}^{2n+1}$. Since $d\theta_0$ is the standard Kähler form of $\C^{n+1}$, the induced Levi form on $H(\Sp^{2n+1})$ coincides with the standard metric of the sphere.

\smallskip
If  $\textsl{V}^{p,q}$ is the space of harmonic polynomials of bi-degree $(p,q)$ in $\mathbb{C}^{n+1}$, then,  $\xi$ acts on $\textsl{V}^{p,q}$  as the multiplication by $i(p-q)$ and it can be deduced, using \eqref{green},
that  $\textsl{V}^{p,q}$ is an eigenspace of $\Delta_{\theta_0}$ on $\mathbb{S}^{2n+1}$ with eigenvalue $2n(p+q)+4pq$ (see \cite[Theorem 4.1]{Cowling} and \cite[Proposition 4.4]{NS} for details). 
 Therefore, 
\begin{equation}\label{eqn e44}
\lambda_1(\theta_0)=2n.
\end{equation}

\subsection{One-parameter groups of CR-automorphisms of the sphere}

 A differentiable map $\varphi:M\rightarrow \widetilde{M}$ between two CR manifolds is a CR map if for any $x\in M$,
\begin{equation}\label{crmap}
d_x \varphi (H_x(M))\subset H_{\varphi (x)}(\widetilde{M})
\qquad 
\mbox{and}
\qquad 
d_x \varphi \circ J^M_x = J^{\widetilde{M}}_{\varphi (x)}\circ d_x \varphi .
\end{equation}
A CR-automorphism of a CR manifold $M$ is a diffeomorphism of $M$ which is a CR map.

\smallskip

Let $e_{n+1}=(0,\cdots,0,1)\in \mathbb{S}^{2n+1}$. The punctured sphere $\mathbb{S}^{2n+1}\setminus \{e_{n+1}\}$ can  be identified with the boundary of the so-called Siegel domain  $\Omega_{n+1}=\{(z,w)\in \C^n\times \C \ :\   \mbox{Im}\ w>\vert z\vert ^2 \}\subset  \C^{n+1}$   through the CR diffeomorphism $ \Phi : \mathbb{S}^{2n+1}\setminus \{e_{n+1}\}\to \partial \Omega_{n+1}$ given by
$$\Phi(\zeta)=\frac 1 {1-\zeta_{n+1}} \left(\zeta_1,\cdots,\zeta_n, i( 1+\zeta_{n+1})\right)$$
with
$$\Phi^{-1}(z,w) =  \frac 1{w+i} (2i z_1,\cdots,2 i z_n, w-i)$$
For every  $t \in \R$, the ``dilation" 
$$\begin{array}{ccccc}
H_t & : &  \partial\Omega_{n+1} & \to &  \partial\Omega_{n+1}\\
 & & (z,w)& \mapsto & (e^t z, e^{2t} w) \\
 \end{array}$$
is a CR-automorphism of $ \partial \Omega_{n+1}$. 
We define $\gamma_t: \mathbb{S}^{2n+1}\to\mathbb{S}^{2n+1}$ by $\gamma_t(e_{n+1})=e_{n+1}$ and, $\forall \zeta\in  \mathbb{S}^{2n+1}\setminus \{e_{n+1}\}$,
$$\gamma_t (\zeta) = \Phi^{-1}\circ H_t\circ \Phi (\zeta) =  \frac1{\cosh t +\sinh t \, \zeta_{n+1}} \left(\zeta_1, \cdots,\zeta_n, \sinh t +\cosh t  \ \zeta_{n+1} \right)$$
or 
$$\gamma_t (\zeta) = \frac1{\cosh t +\sinh t \, \zeta_{n+1}} \left(\zeta + \left(  \sinh t +(\cosh t-1)  \ \zeta_{n+1}\right)e_{n+1}\right).$$
\begin{lemma}
For every $t,$ the map $\gamma_t$ is a CR-automorphism of $\mathbb{S}^{2n+1}$ which satisfies $$(\gamma_t)^*\theta_0=\frac{1}{\left|\cosh t +\sinh t \, \zeta_{n+1}\right|^2}\,\theta_0 .$$
\end{lemma}
\begin{proof}
Let $f_j(\zeta)=\frac{\zeta_j}{\cosh t +\sinh t \, \zeta_{n+1}}$, $j\le n$, and  $f_{n+1}(\zeta)=\frac{\sinh t +\cosh t  \, \zeta_{n+1}}{\cosh t +\sinh t \, \zeta_{n+1}}$.
Then 
$$df_j=\frac{d\zeta_j}{\cosh t +\sinh t \, \zeta_{n+1}}-\frac{ \sinh t\, \zeta_j  \,d\zeta_{n+1}}{(\cosh t +\sinh t \,\zeta_{n+1})^2} \ \ \mbox{and}\ \ 
df_{n+1}=\frac{d \zeta_{n+1}}{(\cosh t +\sinh t \, \zeta_{n+1})^2}.$$
Therefore
$$f_j\, d\bar{f}_j=\frac{\zeta_j\,d\bar{\zeta}_j}{\left|\cosh t +\sinh t \, \zeta_{n+1}\right|^2}-\frac{\left|\zeta_j\right|^2\sinh t \, d\bar{\zeta}_{n+1}}{\left|\cosh t +\sinh t \, \zeta_{n+1}\right|^2(\cosh t +\sinh t \, \bar{\zeta}_{n+1})},$$
$$\bar{f}_j \, df_j=\frac{\bar{\zeta}_{j}\,d\zeta_j}{\left|\cosh t +\sinh t \, \zeta_{n+1}\right|^2}-\frac{\left|\zeta_j\right|^2 \sinh t \, d\zeta_{n+1}}{\left|\cosh t +\sinh t \, \zeta_{n+1}\right|^2(\cosh t +\sinh t \, \zeta_{n+1})}$$
which gives with $\sum_{j=1}^n\left|\zeta_j\right|^2=1- \left|\zeta_{n+1}\right|^2$, 
$$\sum_{j=1}^n \left( f_jd\bar{f_j}-\bar{f_j}df_j \right)=\frac 1{\left|\cosh t +\sinh t \, \zeta_{n+1}\right|^2}\sum_{j=1}^n \left(\zeta_j\,d\bar{\zeta_j}-\bar{\zeta_j}\,d\zeta_j \right)+ \qquad \qquad\qquad \qquad$$ 
$$\qquad \qquad \qquad  \frac{ \left(1-\left|\zeta_{n+1}\right|^2 \right) \sinh t}{\left|\cosh t +\sinh t \, \zeta_{n+1}\right|^2}  \left( \frac{d\zeta_{n+1}}{\cosh t +\sinh t \, \zeta_{n+1}} -\frac{d\bar{\zeta}_{n+1}}{\cosh t +\sinh t \, \bar{\zeta}_{n+1}}  \right).$$
On the other hand, 
$$
f_{n+1}\, d\bar{f}_{n+1}-\bar{f}_{n+1}\, d f_{n+1} =  \qquad \qquad\qquad \qquad\qquad \qquad \qquad\qquad\qquad \qquad \qquad\qquad\qquad$$
$$ \frac 1{\left|\cosh t +\sinh t \, \zeta_{n+1}\right|^2}\left( 
\frac{\sinh t +\cosh t  \ \zeta_{n+1}}{\cosh t +\sinh t \, \bar{\zeta}_{n+1}}\ d\bar{\zeta}_{n+1}
-\frac{(\sinh t +\cosh t  \, \bar{\zeta}_{n+1})}{\cosh t +\sinh t \, \zeta_{n+1}} \, d\zeta_{n+1} \right)
$$
Now (with $\left|\zeta_{n+1}\right|^2=  \zeta_{n+1}\bar{\zeta}_{n+1}$),
$$\frac{ \left(1-\left|\zeta_{n+1}\right|^2 \right) \sinh t -\left( \sinh t +\cosh t  \, \bar{\zeta}_{n+1}\right)}{\cosh t +\sinh t \, \zeta_{n+1} } = -   \bar{\zeta}_{n+1}$$
and
$$\frac{ -\left(1-\left|\zeta_{n+1}\right|^2 \right) \sinh t +\left( \sinh t +\cosh t  \, {\zeta_{n+1}}\right)}{\cosh t +\sinh t \, \bar{\zeta}_{n+1} } =  {\zeta_{n+1}}$$
Thus,
$$\sum_{j=1}^{n+1} \left( f_jd\bar{f_j}-\bar{f_j}df_j \right)=   \frac{1}{\left|\cosh t +\sinh t \, \zeta_{n+1}\right|^2}\sum_{j=1}^{n+1} \left(\zeta_j\,d\bar{\zeta_j}-\bar{\zeta_j}\,d\zeta_j \right)$$
that is,
$$(\gamma_t)^*\theta_0= \frac i 2\sum_{j=1}^{n+1} \left( f_jd\bar{f_j}-\bar{f_j}df_j \right)=\frac{1}{\left|\cosh t +\sinh t \, \zeta_{n+1}\right|^2} \,\theta_0.$$
\end{proof}
Let $p\in \mathbb{S}^{2n+1}$ be any point of the sphere and let $\alpha_p\in U(n+1)$ be such that $\alpha_p( p )=e_{n+1}$. 
The family $\gamma_t ^p=\alpha_p^{-1}\circ \gamma_t\circ \alpha_p$ is a  1-parameter group of CR-automorphisms of the sphere $\mathbb{S}^{2n+1}$ with
\begin{equation}\label{e: 2a0}
\gamma_t^p (\zeta) = \frac1{\cosh t +\sinh t \, (\zeta , p)} \left\{\zeta + \left(  \sinh t +(\cosh t-1)  \,  (\zeta , p) \right)p\right\}
\end{equation}
and
\begin{equation}\label{e: 2a}
(\gamma_t^p)^*\theta_{0}=\frac{1}{\left|\cosh t +\sinh t \,  (\zeta , p)\right|^2}\,\theta_{0}.
\end{equation}


\subsection{Preparatory lemmas}
\begin{lemma}\label{lem 0}
Let $M$ be a strictly pseudoconvex CR manifold of dimension $2n+1$ and let $\theta$, $\hat \theta \in {\mathcal P}_+(M)$ be two pseudo-Hermitian structures with $ \hat \theta =f\, \theta $,  $f\in C^\infty(M)$. Then
\begin{equation}\label{equ 32a}
\psi_{\hat{\theta}}=f^{n+1}\psi_{\theta}
\end{equation}
\end{lemma}
\begin{proof}
From $d\hat{\theta}=f\, d\theta+df\wedge \theta$ we deduce, by induction, 
$$(d\hat{\theta})^n=f^n(d\theta)^n+\alpha_n \wedge \theta$$
where $\alpha_n$ is a differential form of degree $2n-1.$ Thus, 
$$
\hat{\theta}\wedge (d\hat{\theta})^n= f\theta \wedge (f^n(d\theta)^n+\alpha_n \wedge \theta)
= f^{n+1} \theta\wedge (d\theta)^n.
$$
\end{proof}
\begin{lemma}\label{lem 2}
Let $M$ be a strictly pseudoconvex CR manifold of  dimension $2m+1$ and let $\phi: M \to (\mathbb{S}^{2n+1},\theta_0)$ be a CR map. Then, for every $\theta\in {\mathcal P}_+(M)$, 
\begin{equation}\label{equ 33}
\phi^*\theta_{0}=\frac{1}{2m}\left(\sum_{i=1}^{2n+2}\left|\nabla^H\phi_i\right|^2_{\theta}\right)\,\theta
\end{equation}
where $\phi_1,\dots,\phi_{2n+2}$ are the Euclidean components of $\phi$.
\end{lemma}
\begin{proof}
Since $\phi$ is a CR map, the $1$-form $\phi^*\theta_{0}$ vanishes on $H(M)$ which implies  
that there exists $f\,\in C^\infty(M)$ such that 
\begin{equation}\label{e:3.1}
\phi^*\theta_{0}=f\theta.
\end{equation}
Differentiating, we get 
$$\phi^*d\theta_{0}=df\wedge \theta+fd\theta.$$
Hence, for every $X$, $Y\in H_x(M),$ one has $\phi^*d\theta_{0}(X,Y)=fd\theta (X,Y)$ and, using \eqref{crmap},
$$L_{\theta_0}(d\phi(X),d\phi(X))= d\theta_0 (d\phi(X),J^{\Sp^{2n+1}}d\phi(X))=d\theta_0 (d\phi(X),d\phi(J^{M}X))$$
$$=\phi^*d\theta_{0}(X, J^{M}X)=fd\theta (X, J^{M}X) = fL_\theta (X,X).$$
On the other hand, since $L_{\theta_0}$ coincides with the standard inner product on $H_{\phi(x)}(\Sp^{2n+1})$, 
$$L_{\theta_0}(d\phi(X),d\phi(X))=\sum_{i=1}^{2n+2} \left( d\phi_i(X)\right)^2=\sum_{i=1}^{2n+2} L_\theta( \nabla^H\phi_i ,X)^2.$$
Thus,
$$ fL_\theta (X,X)=\sum_{i=1}^{2n+2} L_\theta( \nabla^H\phi_i ,X)^2 .$$
Taking an $L_\theta$-orthonormal basis $\{e_1, \dots,e_{2m}\}$ of $H_x(M)$, we get 
$$2mf=\sum_{j=1}^{2m}\sum_{i=1}^{2n+2} L_\theta( \nabla^H\phi_i ,e_j)^2 =\sum_{i=1}^{2n+2}L_\theta (\nabla^H\phi_i,\nabla^H\phi_i)=  \sum_{i=1}^{2n+2}\left|\nabla^H\phi_i\right|_{\theta}^2$$ 
which implies \eqref{equ 33}, thanks to \eqref{e:3.1}. 
\end{proof}
If $\phi:M\to\R^N$ is a map and $\mu$ is a measure on $M$, we denote by $\int_M\phi \, d\mu$ the vector   $\left(\int_M\phi_1d\mu,\dots, \int_M\phi_{N}d\mu \right)\in\R^N$, where $\phi=(\phi_1,\cdots,\phi_N)$.

\begin{lemma}\label{lem 1}
Let $M$ be a  compact  manifold and let $\mu$ be a measure on $M$ such that no open set has measure zero. If $\phi: M \to \mathbb{S}^{2n+1}$ is a  non constant continuous map, then there exists  a   pair $(p,t)\in \mathbb{S}^{2n+1}\times [0,+\infty)$ such that 
$$\int_M \gamma^p_t \circ\phi  \, d\mu=0.$$
\end{lemma}
\begin{proof}
The proof uses standard arguments (see \cite{EI0, Hersch}). We consider the map 
$$\begin{array}{ccccc}
F & : & \mathbb{S}^{2n+1}\times (0,+\infty) & \to & \mathbb{B}^{2n+2}\subset \mathbb{R}^{2n+2}\\
 & & (p,t)& \mapsto & \frac{1}{V}\int_M \gamma^p_t \circ\phi  \, d\mu \\
 \end{array}$$
 where $V=\int_M d\mu$ and $\mathbb{B}^{2n+2}$ is the unit Euclidean ball. Observe that (see \eqref{e: 2a0}), $\forall p\in \mathbb{S}^{2n+1}$, $\gamma^p_0$ is the identity map while, $\forall \zeta\ne -p$,  $\gamma^p_t(\zeta)$ tends to $p$ as $t\to +\infty$. Consequently, 
 $F(\cdot,0)=\frac{1}{V}\int_M \varphi\, d\mu$ is a constant map 
 and $F(\cdot,t)$ tends to the identity of  $\, \mathbb{S}^{2n+1}$ as $t\to +\infty$. Such a map $F$ is necessarily onto which implies that the origin of $\R^{2n+2}$ belongs to its image. 
\end{proof}
\subsection{Proof  Theorem \ref{theo 41} }
Let ${\theta}\in {\mathcal P}_+(\Sp^{2n+1})$ be a strictly pseudoconvex pseudo-Hermitian structure. We apply   Lemma \ref{lem 1} to the identity map of $\Sp^{2n+1}$ and the measure  induced by $\psi_{ {\theta}} $ to obtain the existence 
of a pair  $(p,t)\in \mathbb{S}^{2n+1}\times [0,+\infty)$ such that 
\begin{equation*}
\int_{\mathbb{S}^{2n+1}} \gamma^p_t \, \psi_{ {\theta}}=0.
\end{equation*}
For simplicity, we write $\gamma$ for  $\gamma^p_t $. The Euclidean components $\gamma_1,\dots \gamma_{2n+2}$ of $\gamma$ satisfy $\int_{\mathbb{S}^{2n+1}} \gamma_j\, \psi_{ {\theta}}=0$. Hence, applying the min-max principle \eqref{minmax0}, we get for every $j\le 2n+2$,  
$$\lambda_1({\theta})\int_{\mathbb{S}^{2n+1}} \gamma_j^2\, \psi_{ {\theta}} \le \int_{\mathbb{S}^{2n+1}} \left| {\nabla}^H \gamma_j \right|_{ {\theta}}^2 \psi_{ {\theta}} .$$
 Summing up we obtain, with $\sum_{j\le 2n+2} \gamma_j^2 =1$,
\begin{eqnarray}
\lambda_1({\theta}) V({\theta}) &\leq& 
\sum_{j\le 2n+2}  \int_{\mathbb{S}^{2n+1}} \left| {\nabla}^H \gamma_j \right|_{ {\theta}}^2 \psi_{ {\theta}} \\
\label{minmax}
&\le& 2n\left(  \int_{\mathbb{S}^{2n+1}} \left( \frac 1{2n}\left| {\nabla}^H \gamma_j \right|_{ {\theta}}^2\right)^{n+1} \psi_{{\theta}} \right)^{\frac{1}{n+1}} 
V({\theta})^{1- \frac{1}{n+1}}.
\end{eqnarray}
Using Lemma \ref{lem 2}, we see that, since $\gamma$ is a CR map from $(\mathbb{S}^{2n+1},{\theta})$ to $(\mathbb{S}^{2n+1},{\theta_0})$,
\begin{eqnarray}\label{gamma}
\gamma^*\theta_0= \left(\frac 1{2n}\sum_{j\le 2n+2} \left| {\nabla}^H \gamma_j \right|_{ {\theta}}^2 \right) {\theta}
\end{eqnarray}
which gives, thanks to Lemma \ref{lem 0}
$$\psi_{\gamma^*\theta_0}= \left(\frac 1{2n}\sum_{j\le 2n+2} \left| {\nabla}^H \gamma_j \right|_{ {\theta}}^2 \right)^{n+1} \psi_{{\theta}}.$$
Thus
$$\lambda_1({\theta}) V({\theta}) \le 2n V({\gamma^*\theta_0})^{ \frac{1}{n+1}}V({\theta})^{1- \frac{1}{n+1}}.
$$
Since  $V({\gamma^*\theta_0})=V(\theta_0)$, we finally get
$$\lambda_1({\theta}) V({\theta})^{ \frac{1}{n+1}} \le 2n V({\theta_0})^{ \frac{1}{n+1}} 
$$
which proves the inequality of the theorem thanks to \eqref{eqn e44}.

\smallskip
Now, if the equality holds  in \eqref{maineq}, this means that we have equality in the Cauchy-Schwarz inequality used in \eqref{minmax}. Thus, $\sum_{j\le 2n+2}   \left| {\nabla}^H \gamma_j \right|_{ {\theta}}^2 $ is constant and, thanks to \eqref{gamma}, ${\theta}$ is proportional to ${\gamma^*\theta_0}$ which takes the form given by \eqref{e: 2a}.

\smallskip
Conversely,  it is clear that when $\gamma$ is a CR-automorphism of the sphere then $\lambda_1( {\gamma^*\theta_0})= \lambda_1({\theta_0})=2n $ and  $V({\gamma^*\theta_0})=V(\theta_0)$. Hence, the equality holds in \eqref{maineq}.



\end{document}